\documentclass[11pt]{amsart}

\usepackage{graphicx,color,latexsym}
\usepackage{verbatim}
\usepackage{enumerate}
\usepackage[all,cmtip]{xy}
\usepackage{enumitem}

\usepackage{subfigure}
\usepackage{mathrsfs}
\usepackage{amscd, amssymb, amsmath, amsthm, graphics}
\usepackage{amsmath,amsfonts,amsthm,amssymb}
\usepackage{latexsym,amsmath}
\usepackage{graphicx}

\usepackage{latexsym}
\usepackage{enumerate}
\usepackage{color}

\newtheorem{theorem}{Theorem}[section]
\newtheorem{prop}[theorem]{Proposition}
\newtheorem{lemma}[theorem]{Lemma}
\newtheorem{remark}[theorem]{Remark}

\newtheorem{definition}[theorem]{Definition}
\newtheorem{cor}[theorem]{Corollary}
\newtheorem{conjecture}[theorem]{Conjecture}
\newtheorem{speculation}[theorem]{Speculation}

\newtheorem{lem}[theorem]{Lemma}

\newtheorem{condition}[theorem]{Condition}

\theoremstyle{definition}

\usepackage{xpatch}
\makeatletter
\xpatchcmd{\@thm}{\thm@headpunct{.}}{\thm@headpunct{}}{}{}
\makeatother

\DeclareMathOperator{\rk}{rank}

\newcommand{\wh}[1]{\widehat{#1}}

\newcommand{\bthm}{\begin{theorem}}
\newcommand{\ethm}{\end{theorem}}

\newcommand{\blem}{\begin{lem}}
\newcommand{\elem}{\end{lem}}

\newcommand{\bcor}{\begin{cor}}
\newcommand{\ecor}{\end{cor}}

\newcommand{\bprop}{\begin{prop}}
\newcommand{\eprop}{\end{prop}}

\newcommand{\brmk}{\begin{remark}}
\newcommand{\ermk}{\end{remark}}

\newcommand{\bpf}{\begin{proof}}
\newcommand{\epf}{\end{proof}}

\newcommand{\beq}{\begin{equation}}
\newcommand{\eeq}{\end{equation}}

\numberwithin{equation}{section}

\def\R{\mathbb{R}}

\def\CP{\mathbb{CP}}

\def\w{\omega}

\def\xkm2{\overline{X}_{k-2}}

\begin{document}

\title[]{The Minimal Genus Problem\\ \vspace{0.2cm}\tiny A Quarter Century of Progress\\ \vspace{0.5cm} {\it In honor of Prof. Banghe Li, on the occasion of his  80'th birthday, in recognition of his fundamental contributions to the minimal genus problem. }}

\author{Josef G. Dorfmeister}

\address{Department of Mathematics\\ North Dakota State University\\ Fargo, ND 58102}
\email{josef.dorfmeister@ndsu.edu}

\author{Tian-Jun Li}
\address{School  of Mathematics\\  University of Minnesota\\ Minneapolis, MN 55455}
\email{tjli@math.umn.edu}

\begin{abstract}  This paper gives a survey of the progress on the minimal genus problem since Lawson's \cite{L} survey.
\end{abstract}

\maketitle

\tableofcontents

\section{Introduction}

Let $M$ be a smooth, closed 4-manifold.  Then, for any given class $A\in H_2(M,\mathbb Z)$, there exists a connected, smoothly embedded surface $\Sigma$ with $[\Sigma]=A$.  It is possible to increase the genus of $\Sigma$ to obtain a homologous surface.  On the other hand, determining the smallest possible genus of such a surface representing $A$, the minimal genus $g_m(A)$, is generally very hard.  With the advent of gauge theoretic methods in 4-manifold theory in the 80's and 90's it became possible to determine this genus for certain manifolds.  These results are described in the excellent survey by Lawson \cite{L}.

In the quarter century that has passed since that survey was written, further progress has been made on the minimal genus problem.  Some of these results continue to use gauge theoretic methods; other results begin to answer the question for classes for which Seiberg-Witten theory fails to provide any information.

In general, this problem has been addressed as follows:\begin{enumerate}
\item Submanifolds with an auxiliary structure, for example a complex structure or symplectic structure, are studied, and a generalized version of the Thom conjecture is proven for these structures.  Such submanifolds are known to be genus minimizing.
\item A lower bound on the genus is determined using an adjunction type inequality.  An upper bound is determined through an explicit construction of an embedded, connected surface.
\end{enumerate}

Adjunction-type inequalities, leading to a lower bound on the possible minimal genus, are surveyed in Section \ref{s:adj}.  Wall \cite{Wa2} showed that there cannot be any cohomology adjunction formulas for manifolds with strongly indefinite intersection forms ($b^\pm \ge 2$).  It thus remains to study the case of manifolds with $b^\pm=1$, or those with definite intersection forms.  Section \ref{s:b+1} surveys results for $b^+=1$, largely for surfaces of non-negative self-intersection (or, equivalently $b^-=1$ and non-positive self-intersection); see also the introduction of \cite{DHL}.

The remainder of this survey describes minimal genus determinations obtained for certain classes of manifolds: rational manifolds, ruled manifolds and elliptic surfaces.  A new result for pairs of disjoint surfaces in $b^\pm=2$ definite manifolds is given in  Section \ref{s:new}.

This survey focuses on orientable surfaces in closed manifolds.  Levine-Ruberman-Strle \cite{LRS} studied non-orientable surfaces $\Sigma$ of genus $g$ in four-manifolds $M$.  The main result here concerns a homology cobordism between rational homology spheres, and proves an inequality on the genus of an essential embedding $\Sigma$.  The statement (and proof) of this result involves certain invariants defined by Osv\'ath-Szab\'o \cite{OS3}.  They also obtained adjunction-type results adapted to the non-orientable setting for certain four-manifolds; see also Iida-Mukherjee-Taniguchi (Theorem 4.6, \cite{IMT}) for an adjunction-type result concerning manifolds with a boundary.
Notably, Gompf (\cite{Go}) produces new powerful existence results for exotic smooth structures on open manifolds using a minimal-genus function and its analogue for end homology.

{\bf Notation and Conventions:}  $H_2(M,\mathbb Z)$ and $H^2(M,\mathbb Z)$ will be identified via Poincar\'e duality and will not be distinguished unless crucial.  A class $c\in H_2(M,\mathbb Z)$ is called {\it characteristic} if $c\cdot A=A\cdot A$ (mod 2) for all $A\in H_2(M,\mathbb Z)$ and $\sigma(M)$ denotes the signature of the manifold $M$.  Throughout, unless otherwise stated, homology and cohomology are assumed to be torsion free.

\section{Adjunction-Type Inequalities}\label{s:adj}

In this section we describe some adjunction-type inequalities obtained for certain classes of manifolds; most of these result from Seiberg-Witten theory or perturbations thereof.  

The notion of an adjunction inequality is due to Kronheimer-Mrowka, first appearing in \cite{KM1} and \cite{KM2}.  A strengthened inequality for the Donaldson invariant appeared in \cite{KM3} while a version for the Seiberg-Witten invariants led to the proof of the Thom conjecture in \cite{KM4}.     

\begin{definition}
Let $A\in H_2(M,\mathbb Z)$ and let $\Sigma$ be a connected embedded surface representing the class $A$.  Let $\chi(\Sigma)$ denote the Euler characteristic of this surface.  An {\it adjunction-type inequality} is an inequality of the form
\[
-\chi(\Sigma)\ge A\cdot A + f(A),
\]
where $A\cdot A$ denotes the self-intersection given by the intersection pairing of $M$, and $f(A)$ is a function generally of the following structure:  for the manifold $M$, a special set of classes in $H_2(M,\mathbb Z)$ is determined and evaluated, using the intersection pairing, against $A$, so $f(A)$ is a degree 1 function of $A$ in some way optimal with respect to this set of values.
\end{definition}

\subsection{Ozsv\'ath-Szab\'o's Adjunction Inequalities}

The original Thom conjecture asserts that the genus of an embedded surface $\Sigma$ representing a homology class $A\in H_2(\mathbb CP^2,\mathbb Z)$ is minimized by an algebraic curve.  This conjecture has been proven and subsequently generalized in a number of different directions; see Theorems 12 and 29, \cite{L}.  Ozsv\'ath and Sz\'abo proved a symplectic version as follows:

\begin{theorem}[Thm 1.1, \cite{OS1}](Symplectic Thom Conjecture) An embedded symplectic surface in a closed, symplectic four-manifold is genus-minimizing in its homology class.

\end{theorem}

Previous results in this direction had only shown this to be true for those classes with non-negative self-intersection.  The key argument in their proof is a calculation of certain Seiberg-Witten relations for negative self-intersection surfaces (Thm 1.3, \cite{OS1}).  As an auxiliary result, they obtained an adjunction inequality for negative self-intersection classes (see Theorems 11 and 17, \cite{L} for the non-negative case).

\begin{theorem}[Cor. 1.7, \cite{OS1}] \label{t:sw} Let $M$ be a smooth, closed four-manifold of simple type with $b^+>1$ and $\Sigma$ a smoothly embedded, oriented, closed surface of genus $g>0$ and $[\Sigma]\cdot[\Sigma]<0$.  Then, for all Seiberg-Witten basic classes $s$, 
\[
-\chi(\Sigma)\ge [\Sigma]\cdot[\Sigma] +|\langle [\Sigma], c_1(s)\rangle|.
\]

\end{theorem}

As is generally the case in Seiberg-Witten theory, if $b^+=1$, then this result holds under consideration of the chamber structure; see Remark 1.10, \cite{OS1}.

Similar techniques are applied in \cite{OS2} to refine the adjunction inequality for non-negative self-intersection surfaces.

\begin{remark} In a series of papers, Lambert-Cole claimed a proof of the Thom conjecture \cite{Lam1} and the symplectic Thom conjecture \cite{Lam2} making no use of gauge theory or any pseudoholomorphic techniques.  Instead, the proof used trisections introduced by Gay and Kirby.

Furthermore, in \cite{Lam2} and \cite{Lam3}, a variety of adjunction inequalities were proven which hold without any restrictions on $A\cdot A$, the genus of the curve or $b^+$.  The following Theorem 1.2 in \cite{Lam2} is one such example: 

Let $(M,\omega)$ be a close, symplectic four-manifold with $[\omega]$ integral.  If $\Sigma$ is a smoothly embedded surface with $[\Sigma]\cdot [\omega]>0$, then
\[
-\chi(\Sigma)\ge [\Sigma]\cdot[\Sigma] - \langle [\Sigma], c_1(\omega)\rangle.
\]

Unfortunately, it was discovered in 2022 that his proof contains a flaw (\cite{Lam4}, \cite{Mro}).  
\end{remark}
%

\subsection{Symplectic Genus}

The Symplectic Thom Conjecture illustrates the role of symplectic surfaces in the minimal genus problem.  It thus seems natural that any class which satisfies the homological necessary restrictions to be represented by a symplectic surface should allow for some restrictions on the possible genus arising from the symplectic structures on $M$. 
The symplectic genus defined by B-H.Li - T-J.Li (Def. 3.1, \cite{LL2}) offers such a bound.  

Furthermore, Theorem \ref{t:sw} requires knowledge of Seiberg-Witten basic classes.  In the symplectic setting, Taubes showed (see Theorem 6, \cite{L}) that the symplectic canonical class $K_{\omega}$ is a SW-basic class for any symplectic form $\omega$.

For a class $A\in H_2(M,\mathbb Z)$ of a symplectic manifold $(M,\omega)$ to be represented by a symplectic surface, it must hold that $[\omega]\cdot A>0$.  In other words, there must exist $\alpha\in\mathcal C_M$ in the symplectic cone with $\alpha\cdot A>0$.  Let $\mathcal K$ denote the set of symplectic canonical classes and $\mathcal C_{M,K}\subset \mathcal C_M$, $K\in\mathcal K$, the set of classes representable by symplectic forms $\omega$ with $K_\omega=K$.  Consider the following set:
\[
\mathcal K_A=\{K\in\mathcal K\;|\;\exists \alpha\in\mathcal C_{M,K}:  \alpha\cdot A>0\}.
\]
For each $K\in \mathcal K_A$, define $\eta_K(A)=\frac{1}{2}(K\cdot A+A\cdot A)+1$.  

\begin{definition}  The {\it symplectic genus} $\eta(A)$ is defined to be $\eta(A)=\max\limits_{K\in\mathcal  K_A}\eta_K(A).$

\end{definition}

Note that there is no guarantee that $\eta(A)\ge 0$.  If $K\in\mathcal  K_A$ is
some symplectic canonical class such that $\eta(A)=\eta_K(A)$, we obtain the inequality
$
\tilde K\cdot A\le K\cdot A
$
for any $\tilde K\in\mathcal  K_A$.  Therefore, define the set \[
\mathcal K(A)=\{K\in\mathcal K_A\;|\;K\cdot A\ge \tilde K\cdot A\mbox{  for all }\tilde K\in\mathcal K_A\}
\]
and note that $\eta(A)=\eta_K(A)$ for any $K\in\mathcal K(A)$.  

\begin{lemma}[Lemma 3.2, \cite{LL2}] \label{l:sg} Let $M$ be a smooth, closed, oriented 4-manifold with non-empty symplectic cone.  Let $A\in H_2(M,\mathbb Z)$.  Then the following hold:

\begin{enumerate}
\item The symplectic genus $\eta(A)$ is no larger than the minimal genus of $A$.
\item The symplectic genus is invariant under the action of Diff(M). 
\end{enumerate}

\end{lemma}

Notice that the first condition ensures that the symplectic genus is well-defined and that it also provides an obstruction to the existence of a smooth / symplectic surface.  It also implies the genus inequality $-\chi(\Sigma)\ge A\cdot A + K\cdot A$
for any class $K\in \mathcal K(A)$.

Moreover, if $A$ is represented by a connected symplectic surface, then the minimal genus and the symplectic genus coincide.

\begin{theorem}[Theorem A, \cite{LL2}] \label{t:symp} Let $M$ be a smooth, closed, oriented four-manifold with a non-empty symplectic cone and $b^+=1$.  Then the symplectic genus of any class of positive square is non-negative and it coincides with the minimal genus for any sufficiently large multiple of such a class.

\end{theorem}

Furthermore, it is shown that the bound given by the symplectic genus is sharp in a broad class of manifolds, in particular for $M=\mathbb CP^2\#k\overline{\mathbb CP^2}$ and $k\le 9$ and $S^2\times S^2$ (see Section 5, \cite{L}; Theorems \ref{t:rat} and \ref{t:irrat}, Lemma \ref{l:sh1}).

If the classes are allowed to have negative square, then it can be shown \cite{D} that there exist examples which have negative symplectic genus.  Hence, those classes cannot be represented by symplectic surfaces and the symplectic genus also does not appear to give any useful bound on the minimal genus.

\subsection{Results for manifolds with $b^+=1$}\label{s:b+1}

The symplectic genus is defined for any symplectic four-manifold.  However, the strongest results are obtained for manifolds with $b^+=1$; see Theorem \ref{t:symp}.  It is thus natural to study results in this context.

\subsubsection{Strle's Results}

Strle obtained the following adjunction-type inequality using Seiberg-Witten theory for manifolds with cylindrical ends:

\begin{theorem}[Theorem A, \cite{S}] \label{t:st} Let $M$ be a smooth, closed, oriented four-manifold with $b_1=0$ and $b^+=1$.  If $\Sigma$ is a smoothly embedded surface of positive self-intersection, then
\[
-\chi(\Sigma)\ge [\Sigma]\cdot[\Sigma] -|\langle [\Sigma], c\rangle|
\]
for any characteristic class $c\in H_2(M,\mathbb Z)$ which satisfies $c\cdot c >\sigma(M)$. 
\end{theorem}

%

The bounds obtained by Strle depend only on the rational homology type of the manifold.  In particular, the bounds can be shown to be sharp in $\mathbb CP^2\#k\overline{\mathbb CP^2}$ and $S^2\times S^2$, but examples are known where the lower bound is not realized; see the Remarks following Theorem 10.1, \cite{S}.

\subsubsection{Cohomological Genus}

Dai-Ho-Li defined the cohomological genus of a class $A$ in purely cohomological terms and obtained lower bounds on the minimal genus of a surface $\Sigma$ representing $A$ purely in terms of the cohomology ring of the 4-manifold.  This generalizes Strle's result.

For the purpose of this survey, we describe only the setup necessary to describe the result in the context of four-manifolds; algebraic details can be found in \cite{DHL}.  

Let $M$ be a smooth, closed, connected, oriented four-manifold with $b^+=1$.  Consider the skew-symmetric bilinear form
\begin{equation}\label{e:T}
T:H^1(M,\mathbb Z)\times H^1(M,\mathbb Z)\rightarrow H^2(M,\mathbb Z),
\end{equation}
and denote that $\tilde b_1=$ rank of $T$.  Define the modified Euler number $\tilde \chi(M)=2+b_2(M)-2\tilde b_1(M)$.  A class $c\in H_2(M,\mathbb Z)$ is called an {\it adjunction class} if it is characteristic and if one of the following holds:
\begin{enumerate}
\item $c\cdot c>\sigma(M)$;
\item $c\cdot c\ge 2\tilde \chi(M)+3\sigma(M)$ and $c$ pairs non-trivially with Im $T$ when $T$ is non-trivial.
\end{enumerate}

\begin{definition}
Let $A\in H_2(M,\mathbb Z)$.  For any adjunction class $c$, define
\[
h_c(A)=\left\{\begin{array}{ll} 1+\frac{A\cdot A-|c\cdot A|}{2}& \mbox{if  }A\ne 0\\0& \mbox{if  }A=0.\end{array}\right.
\]
Define the cohomological genus $h(A)$ of $A$ by $h(A)=\max_c h_c(A)$.
\end{definition}

Dai-Ho-Li then obtained a result analogous to Lemma \ref{l:sg}, which generalizes Strle's Theorem \ref{t:st}'

\begin{theorem}[Theorem 1.4, \cite{DHL}] \label{t:cg} Let $M$ be a smooth, closed, connected, oriented four-manifold with $b^+=1$.  Then, for any $A\in H_2(M,Z)$ with $A\cdot A\ge 0$, the following hold:

\begin{enumerate}
\item The cohomological genus $h(A)$ is no larger than the minimal genus of $A$.
\item The cohomological genus $h(A)$ is invariant under the action of Diff(M). 
\end{enumerate}

\end{theorem}

This again implies the genus inequality $
-\chi(\Sigma)\ge A\cdot A  - |c\cdot A|
$ 
for any adjunction class $c$.  Note that this does not involve any classes related to Seiberg-Witten invariants.

\begin{theorem}[Theorem 1.5, \cite{DHL}]  Let $M$ be a smooth, closed, connected, oriented four-manifold with $b^+=1$ and assume that  $2\tilde \chi(M)+3\sigma(M)\ge 0$.  Then $h(A)\ge 0$ for any $A\cdot A>0$ or $A\cdot A=0$ and $A$ is primitive.

\end{theorem}

It was also shown that the cohomological genus, under the condition $2\tilde \chi(M)+3\sigma(M)\ge 0$, offers a sharp bound for the minimal genus when $M$ is the connected sum of a manifold $Y$ of Kodaira dimension $-\infty$ or $0$ with an appropriate number of copies of $S^1\times S^3$.  The list of possible $Y$ is given explicitly in Theorem 1.5, \cite{DHL}.

For when $2\tilde \chi(M)+3\sigma(M)< 0$, Dai-Ho \cite{DH} proved a partial result with regard to the sharpness of the bound under certain cohomological restrictions on $A$.  These restrictions enabled them to use the symplectic genus to obtain a bound for the comhomological genus.

\subsection{Estimates for Configurations of Surfaces}

Theorem \ref{t:st} is only valid for $b^+=1$, but the techniques used to prove it are readily able to prove a more general result for configurations of $n$ surfaces in manifolds with $b^+=n$.

\begin{theorem}[Theorem B, \cite{S}; Corollary 2.15, \cite{K}; \cite{DHL}]\label{t:st2}
Let $M$ be a smooth, closed, connected four-manifold with $b_1=0$ and $b^+=n>1$. Let $\Sigma_1,...,\Sigma_n$ be disjoint embedded surfaces in $M$ with positive self-intersections.  If $c\in H_2(M,\mathbb Z)$ is a characteristic class satisfying 
\[
c\cdot c>\sigma(M)\mbox{  and  }c\cdot [\Sigma_i]\ge 0\mbox{  for all }i,
\]
 then 
\[
-\chi(\Sigma_i)\ge [\Sigma_i]\cdot[\Sigma_i] -\langle [\Sigma_i], c\rangle
\]
holds for at least one $i$.

\end{theorem}

Konno (Theorem 2.11, \cite{K}) proved an analogous result for classes with vanishing self-intersection.  However, that result is only valid under a homological restriction (Condition 1, \cite{K}).  On the other hand, Konno was able to remove the condition that $b_1=0$.

In fact, the results of Strle, Konno and Dai-Ho-Li are very similar, although the techniques employed vary slightly.  All employ Seiberg-Witten invariants:  in Strle's case for manifolds with cylindrical ends, in Konno the wall-crossing formulas are employed for a fixed spin c structure while two surfaces are considered, and in Dai-Ho-Li the surface is fixed and the two spin c structures are considered.

\subsection{4-Manifolds admitting a Circle Action}

The study of the minimal genus problem for 4-manifolds admitting a circle action involves study of the corresponding problem for 3-manifolds.  The complexity and Thurston norm defined for 3-manifolds below is very similar to the configuration results described in the previous section.  For a survey of the minimal genus problem in 3-manifolds, see Kitayama \cite{Ki} and Wu \cite{W}.

\begin{definition}Let $\Sigma$ be a compact surface with connected components $\Sigma_1,...,\Sigma_n$.  Then, 
\begin{enumerate}
\item  define the {\it complexity of $\Sigma$} to
$
\chi_-(\Sigma)=\sum\limits_{i=1}^n \max(0,-\chi(\Sigma_i));
$
\item let $A\in H_2(M^k,\mathbb Z)$ ($\dim M^k=k $) and define
\[
x_k(A)=\min\{\chi_-(\Sigma)\;|\;\Sigma\subset M\mbox{ an embedded surface representing }A\}.
\]

\end{enumerate}
\end{definition} 

Given a 3-manifold $N$, the quantity $x_3(A)$ is known as the {\it Thurston norm}, although it only defines a seminorm on $H^1(N,\mathbb Z)\equiv H_2(N,\mathbb Z)$, which can however be extended uniquely to $H^1(N,\mathbb R)$.  This can be viewed as a generalization of the knot genus; see Kitayama \cite{Ki} for details.

Observe that the complexity may involve a configuration of surfaces, thus any bound on these quantities is related to the results of Strle and Konno, though it does not account for spheres.

A series of papers by  Kronheimer (\cite{Kr}, \cite{Kr2}), Friedl-Vidussi \cite{FV} and Nagel \cite{N} obtained the following result:

\begin{theorem}[Cor. 7.6, \cite{Kr2}; Theorem 1.1, \cite{FV}; Theorem 5.6, \cite{N}]\label{t:nag}

Let $N$ be an irreducible, closed, oriented 3-manifold.  Assume that $N$ is not a Seifert fibered space and is not covered by a torus bundle.  Let $p:M\rightarrow N$ be an oriented circle bundle.  Then each class $A\in H_2(M,\mathbb Z)$ satisfies
\[
x_4(A)\ge |A\cdot A|+x_3(p_*A).
\]

\end{theorem}

The key to the proofs in \cite{FV} and \cite{N} is to study the Seiberg-Witten invariants of an appropriate finite cover.  In the case of graph manifolds, this necessitates finding a cover where the Turaev norm and the Thurston norm coincide, provding enough basic classes for estimates to work.

It is is also worth noting that this inequality is known to be sharp in cases where the classic adjunction inequality is not; see Section 4.1, \cite{FV}.

\subsection{Overview of Adjunction-Type Results}

In this section, we endeavour to give an overview of the adjunction-type inequalities presented above and in Lawson's survey.

\subsubsection{Strongly indefinite Manifolds}  Let $M$ be a smooth, closed four-manifold with $b^+\ge 2$ or $b^-\ge 2$.  

\begin{enumerate}[leftmargin=*]

\item If $M$ admits non-trivial Gauge theory invariants (including symplectic manifolds), then an adjunction-type inequality holds:
\[
-\chi(\Sigma)\ge A\cdot A + f(A).
\]
Here $f(A)$ is a degree 1 function depending on the Gauge invariants; see Theorem \ref{t:sw}.  We also refer readers to Theorem 22, \cite{FKM}, which proves an adjunction inequality for oriented closed spin 4-manifold which has the same rational cohomology ring as $K^3\#K^3$ using the Bauer-Furuta stable cohomotopy SW invariant; see also Theorem 5.4, \cite{FKMM}.

\item If $M$ admits no non-trivial Gauge invariants, then, only weak adjunction inequalities
\[
-\chi(\Sigma) \ge\alpha A\cdot A+ f(A)
\] 
with $\alpha<1$ exist.

\begin{itemize}[leftmargin=*]
\item Characteristic Classes. Combining the results of Acosta (see Theorems 38 and 39, \cite{L}), Hamilton \cite{H2} and Yasuhara (See Theorem 36, \cite{L}), one obtains the following result:  if $M$ is closed simply connected, $A\in H_2(M,\mathbb Z)$ is characteristic, and either $A\cdot A<0$ (Hamilton) or $A\cdot A$ is equivalent $\sigma(M)$ (Acosta, Yasuhara), then for any surface $\Sigma$ representing $A$, we have a weak adjunction inequality with $\alpha=1/4$.

\item  Divisible Classes. Suppose that $A=2^qB$.  Then the work of Rokhlin \cite{R} produces a weak adjunction-type inequality with $\alpha=1/4$.
Bryan (see Theorem 43, \cite{L}, and also Kotschick-Matic, Theorem 42, \cite{L}) improved this slightly by assuming that $B$ is characteristic and $q$ is small, for example when $q=1$,  $\alpha=15/16$.

If $A=dB$, $d$ is not divisible by 2, then a weak adjunction-type inequality
exists with $\alpha\le \frac{1}{4}$.  More precisely,
\begin{itemize}
\item Rokhlin \cite{R}:  $d$ is a power of an odd prime, so $\alpha=\frac{1}{4}\frac{d^2-1}{d^2}$, which is asymptotic to $\frac{1}{4}$.
\item Bryan (Theorem 43, \cite{L}): under the assumption that $M$ is spin, $
\alpha=\frac{5}{4}\frac{d^2-1}{6d(d-1)}$. 
\end{itemize}
In contrast to these results, in a series of papers, Lee-Wilczynski determined the minimal genus for topologically locally flat embeddings in simply connected four-manifolds; see Theorem 1.1, \cite{LW1} (a general estimate for arbitrary four-manifolds is given in Theorem 2.1, \cite{LW2}).  In the closed, simply connected case, this theorem can be stated as follows:

\begin{theorem}[Theorem 1.1, \cite{LW1}]  Let $M$ be a closed, oriented, simply connected four-manifold.  Suppose that $A=dB\in H_2(M,\mathbb Z)$ is divisible with divisibility $d$.  Then there exists a simple, topologically locally flat embedding $\Sigma$ representing $A$ by an oriented surface of genus $g>0$ if and only if
\[
b_2+2g\ge \max_{0\le j\le d}|\sigma(M)-2j(d-j)B\cdot B|.
\]

\end{theorem}  

If $b^+=1$, and either $A\cdot A\ge 0$ or $d$ is large enough, then this result leads to a weak adjunction-type inequality with $\alpha=1/2$.
Observe that this Theorem actually guarantees the existence of a surface if the inequality is satisfied.

\item  Primitive and Ordinary Classes. There exist manifolds where all such classes are represented by spheres; see Wall \cite{Wa2}.

\item Configurations of Disjoint Surfaces. An adjunction-type inequality 
with $f(A)$ a degree 1 function, which is some cohomological evaluation on $A$, holds in certain settings for at least one of the surfaces; see Theorem \ref{t:st2}.

\end{itemize}

\end{enumerate}

\subsubsection{ $b^+=1$, Classes with Non-Negative Self-Intersection}   An adjunction-type inequality holds
with $f(A)$ a degree 1 function, which is some cohomological evaluation on $A$; see Theorems \ref{t:st} and \ref{t:cg}.

The results also hold if $b^-=1$ and considering classes with non-positive self-intersection, with appropriate absolute values.

\subsubsection{Definite Manifolds with $b^\pm\ge 2$}
\begin{enumerate}
\item Divisible or Characteristic Classes.  The results in the strongly indefinite case all continue to hold and weak adjunction-type inequalities are known.

\item Primitive and Ordinary Classes.  The authors are only aware of a result for configurations of $b^+$ disjoint surfaces; an adjunction-type inequality 
with $f(A)$ a degree 1 function, which is some cohomological evaluation on $A$, holds in certain settings for at least one of the surfaces; see Theorem \ref{t:st2}.
\end{enumerate}

\section{Minimal Genus}

To determine, or bound from above, the minimal genus of a specific class using an adjunction type inequality, it is necessary to explicitly construct a submanifold.  Two key concepts are discussed in the following examples:
\begin{itemize}[leftmargin=*]
\item There are a number of classical constructions such as the connected sum and smoothing of algebraic curves.  A further method, the circle sum introduced in \cite{LL3}, is described below.
\item Reducing the set of classes that need to be studied by using the induced action of the orientation preserving diffeomorphisms of $M$ on $H_2(M,\mathbb Z)$.   
\end{itemize}

\subsection{Rational Manifolds} 

Rational manifolds are either $S^2\times S^2$ or $\mathbb CP^2\#k\overline{\mathbb CP^2}$.  The minimal genus problem for $S^2\times S^2$ and the $k=1$ case was described in Theorem 15, \cite{L}.  The case $k=0$ is the classical Thom conjecture (Theorem 12, \cite{L}).  Therefore, let $M=\mathbb CP^2\#k\overline{\mathbb CP^2}$ and assume that $k\ge 2$.

 Call a basis $\{H,E_1,..,E_k\}$ of $ H_2(M,\mathbb Z)$ the {\it standard basis} if  $H\cdot H=1$, $E_i\cdot E_j=0$ for $i\ne j$ and $E_i\cdot E_i=-1$.  In this basis, a class will be denoted $A=aH-\sum_ib_iE_i=(a,b_1,...,b_k)$.  The standard canonical class $K_{st}$ will denote the class $(-3,-1,...,-1)$.

$H_2(M,\mathbb Z)$ can be viewed from the following three different viewpoints with corresponding automorphism (sub) groups:\begin{enumerate}[leftmargin=*]
\item The homology of the manifold $M$.  The geometric automorphism group is given by
  \begin{equation}\label{e:D(M)}D(M)=\{\sigma\in Aut(H_2(M,\mathbb Z)): \sigma=f_*\text{ for some }f\in \textit{Diff}^+(M)\}.\end{equation}

\item An integer lattice $L$ with quadratic form $Q$ (the intersection form on $M$) together with the orthogonal group $O(L)$ of lattice automorphisms preserving $Q$.

\item (\cite{ZGQ}) The lattice $H_2(M,\mathbb Z)$ is the root lattice of a Kac-Moody algebra, $Q$ is the generalized Cartan matrix, and the Weyl group $W$ is the subgroup of $O(L)$ generated by reflections on classes with $Q(x,x)\in\{-1,-2\}$.  A reflection of $B$ on the class $A$ is given by $r_A(B)=B-\frac{2Q(A,B)}{A^2}A$.  Note that for $k\le 9$, this is a hyperbolic Kac-Moody algebra.

\end{enumerate}

Clearly $D(M)\subset O(L)$.  In fact, the following is true:

\begin{theorem} [\cite{FM}, \cite{ZGQ}, \cite{LL2}]
\begin{enumerate}
\item   If $k\le 9$, then $W= D(M)=O(L)$.
\item If $k\ge 10$, then $D(M)$ is a proper subgroup of $W$.
\end{enumerate}
\end{theorem}

Moreover, B.H. Li- T.J. Li \cite{LL5} identified explicit generators of $D(M)$ (see also \cite{Wa} and \cite{ZGQ}).

\begin{definition}
Two classes $A,B\in H_2(M,\mathbb Z)$ are called {\it $D(M)$-equivalent} if there is a $\sigma\in D(M)$ such that $\sigma(A)=B$.  Denote by $O_A$ the orbit of the class $A$ under the action of $D(M)$.
\end{definition}

Homology classes in rational manifolds under $D(M)$-action exhibit two special classes.  The first, reduced, arises as the elements of the fundamental chamber $C$; see \cite{Ka} and \cite{Wa}.  The second are simplified classes, which form the counterpart to the reduced classes; see \cite{LL4}.    

\begin{definition}  Let $A\in H_2(M,\mathbb Z)$.  \begin{enumerate}
\item A class $A=(a,b_1,...,b_k)$ is called reduced if $a\ge 0,\;\;b_1\ge ...\ge b_k\ge 0$ and \begin{enumerate}
\item $a\ge b_1$ ($k=1$),
\item $a\ge b_1+b_2$ ($k=2$) or 
\item $a\ge b_1+b_2+b_3$ ($k\ge 3$).
\end{enumerate}

\item A class $A=(a,b_1,...,b_k)$ is called simplified if $a\ge 0,\;\;b_1\ge ...\ge b_k\ge 0$ and \begin{enumerate}
\item $2a\le b_1+b_2$ ($k=2$) or 
\item $3a\le b_1+b_2+b_3$ ($k\ge 3$).
\end{enumerate}

\end{enumerate}

\end{definition}


It is not hard to see that if $k\le 9$ and $A^2<0$, then $A$ may be equivalent to a simplified class, but not a reduced one.  For $k\ge 10$, $A^2<0$, $A$ may be equivalent to either type.  For each fixed value $A^2=-n$, the set of simplified classes is always finite.

\begin{lemma}[ \cite{Ka}, \cite{L1}, \cite{LL4}, \cite{ZGQ}]\label{l:class}
Let $M=\mathbb CP^2\#k\overline{\mathbb CP^2}$ and $A\in H_2(M,\mathbb Z)$ with $A\ne 0$.  
\begin{enumerate}[leftmargin=*]
\item $A$ is $D(M)$-equivalent to a reduced or simplified class.
\item Each orbit $O_A$ contains either a simplified class or a reduced class, never both.  There is an efficient algorithm to find a representative in each orbit.
\item If $A$ is $D(M)$-equivalent to a reduced class, then this is the unique reduced class in $O_A$.

\end{enumerate}
\end{lemma}

The case for simplified classes is slightly more subtle.

\begin{lemma}\cite{D2} \label{l:main} Let $M=\mathbb CP^2\#k\overline{\mathbb CP^2}$ ($k\ge 3$) and let $A\in H_2(M,\mathbb Z)$ with $A^2=-n<0$.  Assume that $A$ is primitive and that the $D(M)$-orbit $O_A$ of $A$ contains a simplified class.  Then, for each $(k,n)$, $n>0$, the following is true:\begin{enumerate}[leftmargin=*]
\item If $A$ is ordinary, then $O_A$ is the unique primitive ordinary orbit.
\item If $n\equiv_42$ and $A$ is characteristic, then  $O_A$ is the unique primitive characteristic orbit.
\end{enumerate}
\end{lemma}

For $k=2$ this is known to be false by an example of C.T.C. Wall (see \cite{W}). 

These reductions imply that to understand the minimal genus problem for rational manifolds, it suffices to solve the problem on reduced and simplified classes.  

\begin{lemma}[Lemma 3.4, \cite{LL2}; \cite{D}]
If A is reduced, then 
\[
\eta(A)=\frac{1}{2}(K_{st}\cdot A+A\cdot A)+1.
\] 
\end{lemma}

As noted in Theorem \ref{t:symp}, the symplectic genus is positive if $A\cdot A>0$.  However, if $A\cdot A<0$, then it is possible for this quantity to be negative and thus to offer no apparent useful bound on the minimal genus.  

\begin{theorem}[Theorem B, Theorem C, \cite{LL2}]  \label{t:rat}Let $A\in H_2(M,\mathbb Z)$ and assume that $A$ is primitive if $A\cdot A=0$.  If $A\cdot A\ge \eta(A)-1$, then $A$ is represented by a connected, embedded symplectic surface.  
\end{theorem}

Observe that the inequality in the Lemma precisely ensures that the corresponding moduli spaces are non-empty. 

From this, a characterization of certain classes representable by spheres can be obtained.

\begin{theorem}[Theorem 3, \cite{L1}; Theorem 4.2, \cite{LL2}; Theorem 3, \cite{Kik}]  Up to $D(M)$-equivalence, the following are the only classes with non-negative self-intersection which are spherically representable ($k\ge 1$):

\[
2H;\;\; k(H-E_1); \;\; (k+1)H-kE_1; \;\;(k+2)H-(k+1)E_1-E_2.
\]

\end{theorem}

For when $k\le 9$, the reduced case has been resolved completely.

\begin{lemma} [Theorem 1, \cite{L1}] \label{l:sh1} Assume that $A$ is reduced and $2\le k\le 9$.  Then 
\[
g_m(A)=\max\{0,\eta(A)\}.
\]
\end{lemma}

B.H.-Li \cite{L1} also produced an interesting result for multiple classes.

\begin{lemma}[Prop. 5, \cite{L1}]  If a reduced class $A$ with $A\cdot A\ge 0$ has $g_m(A)=\eta(A)$, then, for any $d>0$, $g_m(dA)=\eta(dA)$.
\end{lemma}

For when $k\ge 10$, there are still very few results.  One result, obtained by B.H-Li and then used by Zhao-Gao-Qiu \cite{ZGQ} to develop an explicit formula, is the following:

\begin{lemma}[Prop. 4, Prop. 6, \cite{L1}] Let $A\ne 0$ be a reduced class with $b_i\le 2$ for $i\ge 10$.  Then $g_m(A)=\eta(A)$ if either
\begin{enumerate}
\item  $\eta(A)=1$, or
\item $\eta(A)\ge 2$, $A\cdot A\ge 0$, and $b_{10}>0$.
\end{enumerate}

\end{lemma}

Note the restriction on the coefficients $b_i$ for $i\ge 10$.

In the simplified case, the key is the orbit structure from Lemma \ref{l:main} and Theorem 2 in \cite{LiLi}, as well as an explicit construction.

\begin{theorem}[\cite{LL4}, \cite{D2}] Let $M=\mathbb CP^2\#k\overline{\mathbb CP^2}$ ($k\ge 3$) and $A\in H_2(M,\mathbb Z)$ with $A^2<0$.  Assume that $A$ is primitive ordinary and that the $D(M)$-orbit $O_A$ of $A$ contains a simplified class; then it has minimal genus $g_A=0$.
Moreover, for each class, there exists a complex orientation-compatible structure on $M$ such that the surface can be chosen to be holomorphic with a symplectic orientation-compatible structure on $M$ such that the surface is symplectic.

\end{theorem}

If $A$ is primitive characteristic, things are a bit more complicated.  In particular, the minimal genus can only be determined for sufficiently many blow-up; which is a stabilization-type result.

\begin{theorem}  [\cite{LL4}, \cite{D2}]  Let $M=\mathbb CP^2\#k\overline{\mathbb CP^2}$ ($k\ge 3$) and $A\in H_2(M,\mathbb Z)$ with $A^2<0$.  Assume that $A$ is primitive characteristic, that $A^2=-(8\gamma+k-1)$ ($\gamma>0$) and that the $D(M)$-orbit $O_A$ of $A$ contains a simplified class.
\begin{enumerate}
\item If $1\le k\le 2\gamma-1$, then the minimal genus of any embedded surface representing $A$ is bounded above by $2\gamma-k$.
\item If $\gamma$ is odd and $k\ge 2\gamma-1$, then the minimal genus of any embedded surface representing $A$ is 1.
\item If $\gamma$ is even and $k\ge 2\gamma$, then the minimal genus of any embedded surface representing $A$ is 0.
\end{enumerate}
\end{theorem}

The explicit classes in each category for $0>A\cdot A>-17$, $3\le k\le 9$ were calculated in \cite{LL2} and \cite{LL4}.  This was extended to all simplified classes in \cite{D2}.

Reduced classes with $A\cdot A<0$ are not understood at the current time.

\subsection{Irrational Ruled Manifolds}

Let $M$ be irrational ruled.  This means that $M$ is diffeomorphic to either a trivial $S^2$-bundle over a surface $\Sigma_g$ of genus $g$, the twisted $S^2$ bundle over $\Sigma_g$, denoted by $N$,  or a blow-up of either of these base cases.  The minimal genus problem for the first two cases is described in Theorems 20 and 21 of \cite{L}.  Blow-ups of the two base cases are diffeomorphic.  

In what follows, assume that $M=N\#k\overline{\mathbb CP^2}$ and $k\ge 1$.  Let $U$ be the section class and $T$ the fiber class, and call the basis $\{U,T, E_1,...,E_k\}$ the standard basis of $H_2(M,\mathbb Z)$ if 
\[
U\cdot U=U\cdot T=-E_i\cdot E_i=1
\mbox{  
and  
}
T\cdot T=U\cdot E_i=T\cdot E_i=E_i\cdot E_j=0
\]
for $i\ne j$.  The standard canonical class $K_{st}$ will denote the class $-2U+(2g-1)T+\sum E_i$.

Consider first the class $A=bT$.  This class always has minimal genus 0 by tubing the b copies of the fiber.  Moreover, the $D(M)$ orbit of this class consists of two elements: $bT$ and $-bT$.  In what follows, this class will be excluded.  The following definition is based on the results in \cite{ZG}:

\begin{definition}For the class $A=aU+bT-\sum_{i=1}^k c_iE_i\in H_2(M,\mathbb Z)$, assume that $A\ne bT$, $k\ge 1$ and 
\[
 c_1\ge ...\ge c_k\ge 0.
\]
 \begin{enumerate}
\item The class $A$ is called {\it reduced} if
\[
a\ge \left\{\begin{array}{ll} c_1 & k=1,\\ c_1+c_2 & k\ge 2.\end{array}\right.
\]

\item The class $A$ is called simplified if $a=0$ and $b$ is the minimal non-negative value in the $D(M)$ orbit containing $A$.
\end{enumerate}  
\end{definition}

It is now possible to mimic the results for rational manifolds.  In particular, analogous results hold for all statements up to Lemma \ref{l:class}.  In particular, reduced and simplified orbits do not overlap, and simplified classes always have negative self-intersection, in a fashion analogous to the rational case.

\begin{lemma}[\cite{LL5}, \cite{ZG}]\label{l:rclass}
Let $M=N\#k\overline{\mathbb CP^2}$, $k\ge 1$  and $A\in H_2(M,\mathbb Z)$ with $A\ne bT$.  
\begin{enumerate}
\item $A$ is $D(M)$-equivalent to a unique reduced or simplified class.
\item Each orbit $O_A$ contains either a simplified class or a reduced class, never both.  There is an efficient algorithm to find a representative in each orbit.

\end{enumerate}
\end{lemma}

 Moreover, the following Theorem continues to hold:

\begin{theorem}[Theorem B, Theorem C, \cite{LL2}] \label{t:irrat} Let $A\in H_2(M,\mathbb Z)$ and assume that $A$ is primitive if $A\cdot A=0$.  If $A\cdot A\ge \eta(A)-1$, then $A$ is represented by a connected, embedded symplectic surface.  If $A\cdot A\ge 0$, then $\eta(A)$ can be calculated using $K_{st}$.
\end{theorem}

Zhao-Gao \cite{ZG} provided an explicit formula for the  reduced class associated to any $A\in H_2(M,\mathbb Z)$.  They then determined the minimal genus for certain classes using this reduction and an explicit construction while making use of the symplectic genus.

\begin{theorem}[Theorem 1.5, Theorem 1.6, \cite{ZG}]  Let $A=aU+bT-\sum_{i=1}^k c_iE_i$ with $A\cdot A\ge 0$.
\begin{enumerate}
\item $A$ is represented by an embedded sphere if and only if $a=0$.
\item If $a\ne 0$, then 
\[
\eta(A)\ge |a|(g-1)+1.
\]
\item If $A$ is reduced, $n\le 4$ and $b\ge 0$, then $g_m(A)=\eta(A)$.
\end{enumerate}

\end{theorem}

B.H.-Li and T.J.-Li proved a similar result in the case $b<0$:

\begin{theorem}[Theorem 3.9, \cite{LL3}]  Let $A=aU+bT-\sum_{i=1}^k c_iE_i$ be a reduced class with $A\cdot A\ge 0$ and $b<0$.  Let $m<k$ be such that $c_m>0=c_{m+1}=...=c_k$.  If either \begin{enumerate}
\item $a+2b\ge \left(\sum_{i=1}^m c_i\right)-m$ or
\item $a+2b=1$, and there is an $m'\le -b$ such that $c_i=2$ for $1\le i\le m'$ and $c_i=1$ for $m'<i\le m$,
\end{enumerate}
then \[g_m(A)=\eta(A).\]

\end{theorem}

The proof of this theorem makes use of a novel construction, called the {\it circle sum}.  This operation takes two surfaces of genus $g_1$ and $g_2$ and produces a new surface of genus $g_1+g_2-1$.  

Let $\Sigma_0$ and $\Sigma_1$ be closed, oriented surfaces of positive genus $g_0$ resp. $g_1$ disjointly embedded in $M$.  By removing a disk from each and tubing them, one obtains the connected sum and a surface of genus $g_1+g_2$.  Instead of considering embedded circles which bound a disk, consider embedded circles $S_i^1\subset \Sigma_i$ which represent a non-trivial class in $H_1(\Sigma_i,\mathbb Z)$.  Remove from each surface an annulus $S^1_i\times [0,1]$ to obtain two surfaces with two boundary circles each.  If the boundaries are pairwise connected by a tube, one either re-obtains the two disjoint surfaces or a new surface is formed.  This is the circle sum of $\Sigma_1$ and $\Sigma_2$, introduced in \cite{LL3}.

The key is to find an appropriate embedded annulus in $M$ that intersects the surfaces $\Sigma_i$ at precisely the circles $S^1_i$, and which admits a non-vanishing normal vector field that is tangential to $\Sigma_i$ but normal to $S^1_i$ at the ends.  Under these conditions, Theorem 2.1, \cite{LL3}, states that the circle sum exists, that the obtained surface has genus $g_1+g_2-1$, and that it is in the class $[\Sigma_0]+[\Sigma_1]$.

\subsection{Elliptic Manifolds}  The results in this section combine the techniques from above.  Some version of an adjunction inequality is used to give a lower bound on the minimal genus, a Diff(M)-argument is used to reduce the set of classes that need to be considered, and finally, the circle sum is used to construct an explicit example, thus determining the minimal genus.

Hamilton \cite{H1} studied minimal, simply connected elliptic surfaces $M$ with $b^+>1$. Let $k$ denote the Poincar\'e dual of the canonical class $K$ of $M$ divided by its divisibility (or $k$ is the fiber class if $M$ is a K3 surface).  Let $V$ denote a class with $k\cdot V=1$ (for example, the class of a section of the fibration if there are no multiple fibers) with $V\cdot V=2b$ or $2b+1$.  Denote the class $W=V-bk$.  Then the intersection form on the span of $W$ and $k$ is either
\[
H=\left(\begin{array}{cc}0&1\\1&0\end{array}\right) \mbox{  if  }V\cdot V=2b
\]
or
\[
H'=\left(\begin{array}{cc}0&1\\1&1\end{array}\right) \mbox{  if  }V\cdot V=2b+1.
\]
Then $M$ has intersection form given by either
\[
H\oplus lH\oplus m(-E_8)\mbox{  or  }H'\oplus lH\oplus m(-E_8),
\]
with $l\ge 2$ and $m>0$.  A (standard) basis for the second $H$ summand can be obtained from the Gompf nucleus $N(2)$; see \cite{GS}.  This nucleus contains a rim torus $R$ of self-intersection 0 and a sphere of self-intersection -2 such that $R\cdot S=1$.  Choose as a basis of the second  $H$ summand the classes $R$ and $T=R+S$, where $T$ can be represented by a torus of self-intersection 0.

Building on work of L\"onne \cite{Lo}, Hamilton obtained the following result:

\begin{lemma}[Prop. 2.10, \cite{H1}]\label{l:hp}
Let $M$ be an elliptic surface and $B$ an arbitrary class in the subgroup $lH\oplus m(-E_8)\subset H_2(M,\mathbb Z)$.  Then $B$ can be mapped to any other class in $lH\oplus m(-E_8)$ of the same self-intersection and divisibility by a  self-diffeomorphism of $M$.  In particular, $B$ can be mapped to a linear combination of $R$ and $S$.  This diffeomorphism acts by identity on the span $\langle k, W\rangle.$  
\end{lemma}

This diffeomorphism allows any class to be mapped to one that is contained in $N(2)$, and thus the following lemma is crucial:

\begin{lemma}[Cor. 3.3, \cite{H1}]  Every non-zero class in $H_2(N(2),\mathbb Z)$, not necessarily primitive, which has self-intersection $2c-2$ with $c\ge 0$, is represented by an embedded surface of genus $c$ in $N(2)$.

\end{lemma}

This result follows from a result by Kronheimer-Mrowka; see Lemma 14, \cite{L}.  Using the previous two results, the following is immediately apparent:

\begin{lemma}[Cor 4.1, \cite{H1}]  Let $M$ be an elliptic surface.  Then every non-zero $A\in H_2(M,\mathbb Z)$ of self-intersection $2c-2$ with $c\ge 0$ that is orthogonal to the classes $K$ and $V$ is represented by a surface of minimal genus $c$.  This surface can be assumed to be embedded in $N(2)\subset M$.
\end{lemma}

A slightly stronger version is possible if $M$ admits no multiple fibers (in this case $k=F$, the class of a fiber).

\begin{theorem}[Theorem 1.1, Theorem 4.8, \cite{H1}]  Let $M$ be an elliptic surface without multiple fibers.  Then every non-zero $A\in H_2(M,\mathbb Z)$ of self-intersection $2c-2$ with $c\ge 0$ that is orthogonal to $K$ is represented by a surface of minimal genus $c$.  
\end{theorem}

The key to this proof is to reduce to the standard class $\alpha F+\gamma R+\delta T$ using Lemma \ref{l:hp} and to use the circle sum construction to produce an embedded surface representing this class.

Nakashima \cite{Nak} studied $T^2$-bundles over surfaces $\Sigma_h$ of genus $h\ge 1$.  To begin, let $M=T^2\times \Sigma_h$.   Denote $x_1, z_1,...,x_h, z_h$ as a symplectic basis of $H_1(\Sigma_h,\mathbb Z)$ as described in Figure \ref{f:nak}.  Similarly, fix a symplectic basis $y,t$ of $T^2$.  Let $S$ and $F$ denote the class of a section and of the fiber in $H_2(M,\mathbb Z)$ respectively.    Write $A\in H_2(M,\mathbb Z)$ as
\[
A=\sum_{i=1}^h(a_i x_i\otimes y+b_i z_i\otimes t+c_i x_i\otimes t + d_i(-z_i\otimes y))+eS+f(-F)=
\]
\[
= (a_1,b_1,c_1,d_1,...,a_h,b_h,c_h,d_h,e,f).
\]
In this basis, the intersection form on $M$ is $H^{\oplus 2g+1}$.  A detailed study of Diff(M) leads to the following result:

\begin{lemma}[Lemma 2.2, \cite{Nak}]  For any homology class $A\in H_2(M,\mathbb Z)$, there exists a self-diffeomorphism of $M$ sending $A$ to the class
\[
(\tilde a_1, \alpha\tilde a_1,0,0,\tilde a_2,0,\tilde c_2,0,...,\tilde a_h,0,\tilde c_h,0,\beta \tilde a_1,\tilde f)
\]
for some $\alpha,\beta\in \mathbb Z$.
\end{lemma}

\begin{figure}[h]
\includegraphics[scale=0.35]{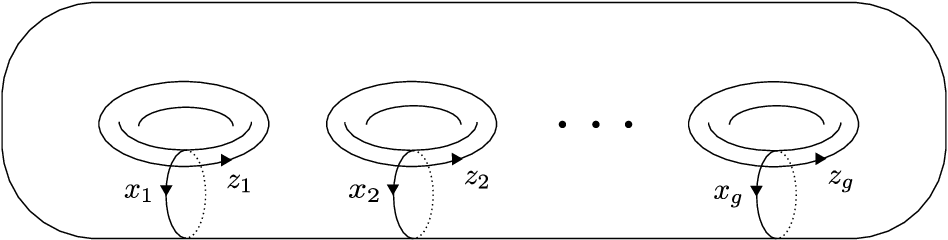}
\caption{The symplectic basis on $\Sigma_h$; see \cite{Nak}.}\label{f:nak}
\end{figure}

Note that the self-intersection is preserved, so the following must hold:
\[
\sum_{i=1}^h(a_ib_i+c_id_i)+ef=\alpha \tilde a_1^2+\beta \tilde a_1 \tilde f.
\]

Using this reduction and the circle sum construction on cleverly constructed parallel tori leads to the following result:

\begin{theorem}[Theorem 1.2, \cite{Nak}]\label{t:n1} Let $M=\Sigma_h\times T^2$ and $A\in H_2(M,\mathbb Z)$, where $\Sigma_h$ is an oriented closed surface of genus $h\ge 1$.  Then 
\[
g_m(A)=\left\{\begin{array}{ll} 0& A=0,\\1+\frac{1}{2}|A\cdot A|+(h-1)|A\cdot F|& (*),\\ 2& \mbox{otherwise,}\end{array}\right.
\]
where $F$ is the fiber class and (*) means that one of the following conditions is satisfied:\begin{itemize}
\item $F\cdot A\ne 0$,
\item $A\cdot A\ne 0$, or
\item $A\ne 0$ and $A=u\otimes v+n(-F)$ for some $u\in H_1(\Sigma_h)$, $v\in H_1(T^2)$ and $n\in \mathbb Z$.
\end{itemize}

\end{theorem}

This result can be re-stated in terms of the Thurston norm by viewing $M$ as $(\Sigma_h\times S^1)\times S^1$ (see Theorem \ref{t:nag}).

\begin{cor}[Cor 4.4, \cite{Nak}]  Let $M=(\Sigma_h\times S^1)\times S^1$ with $h\ge 2$ and let $p:M\rightarrow \Sigma_h\times S^1$ be the projection map.  Then each class $A\in H_2(M,\mathbb Z)$ satisfies
\[
x_4(A)= |A\cdot A|+x_3(p_*A).
\]
\end{cor}

A similar splitting leads to the following result:

\begin{theorem}[Theorem 4.6, \cite{Nak}]  Let $p:N\rightarrow \Sigma_h$ be a non-trivial $S^1$-bundle over $\Sigma_h$ and let $M=N\times S^1$.  Then, for each class $A\in H_2(M,\mathbb Z)$, the minimal genus is given by
\[
g_m(A)=\left\{\begin{array}{ll} 0& A=0,\\1+\frac{1}{2}|A\cdot A|& A\cdot A\ne 0\mbox{  or  }A\ne 0\mbox{  and  } A=u\otimes v+n(-F),\\ 2& \mbox{otherwise.}\end{array}\right.
\]
\end{theorem}

The next result, while not applicable to elliptic fibrations, fits nicely into the structure of Theorem \ref{t:n1}.  Note that the case considered by Edmonds \cite{E}, using only topological methods, and corresponds to the last of the (*) cases with $n=0$.

\begin{theorem}[Theorem 1, \cite{E}]  Let $\Sigma_h$ and $\Sigma_i$ be closed, orientable surfaces of a genus greater than 1.  An element $A\in H_2(\Sigma_h\times \Sigma_i,\mathbb Z)$ is represented by a topological or differentiable embedded torus $\Sigma_1$ if and only if $A=u\otimes v$ for some $u\in H_1(\Sigma_h,\mathbb Z)$ and $v\in H_1(\Sigma_i,\mathbb Z)$.

\end{theorem}

In Example 3.6, \cite{Li}, two disjoint symplectic tori were constructed in $M=\Sigma_2\times\Sigma_2$, representing the classes $A_1$ and $A_2$, respectively.  It was shown that no connected symplectic torus exists in the class $A_1+A_2$, and Edmond's result shows that no smooth torus exists either.  In particular, the circle sum construction cannot be applied to these two surfaces.

These results lead to the following speculation:  

\begin{speculation}  Let $M=\Sigma_h\times \Sigma_i$  and let $A\in H_2(M,\mathbb Z)$.  Then
\[
g_m(A)= \min\{\eta(A),b(A)\}. 
\]
\end{speculation}

The quantity  $b(A)$ is defined as follows (see also Def. 3.4, \cite{Li}, for an equivalent description):  recall the map $T$ from Eq. \ref{e:T},
\[
T:H^1(M,\mathbb Z)\times H^1(M,\mathbb Z)\rightarrow H^2(M,\mathbb Z),
\]  
given by $(\alpha,\beta)\mapsto \alpha\cup\beta$.  Define the map 
\[
B(A):H^1(M,\mathbb Z)\times H^1(M,\mathbb Z)\rightarrow H^2(M,\mathbb Z)
\]
by
\[
B(A)(\alpha,\beta)=T(\alpha,\beta)\cdot A.
\]
Then $b(A)$ is defined as
\[
b(A)=\frac{1}{2}\rk B(A).
\]

\subsection{Disjoint Spheres in Definite Manifolds with $b^+=2$}\label{s:new}

%
%
%

In \cite{La2}, Lawson prove the following Theorem:

\begin{theorem}[Theorem 17, \cite{La2}]\label{t:law} The only characteristic classes in $\mathbb CP^2\#\mathbb CP^2$ which are representable by spheres are $(\pm 1,\pm 1)$.  The only non-trivial divisible classes which can be represented by embedded spheres are
\[
\pm(d,0),\pm(0,d),\pm(2,2),\pm(2,-2)
\]
for $d=2,3$. The primitive and ordinary classes 
\[
\pm(1,0),\;\pm(0,1),\;(\pm 1,\pm 2),\;(\pm 2,\pm 1)
\] 
can be represented by embedded spheres.
\end{theorem} 

\begin{conjecture}[Conjecture 1, \cite{La2}]\label{c:law}
The only primitive and ordinary classes in $\mathbb CP^2\#\mathbb CP^2$ which can be represented by embedded spheres are given by Theorem \ref{t:law}.
\end{conjecture}

Motivated by this Conjecture, let $M$ be a smooth, closed, connected, definite four-manifold with $b_1(M)=0$ and $b^+(M)=2$.  Then $H_2(M,\mathbb Z)\simeq \mathbb Z^2$ and the classes in Theorem \ref{t:law} still make sense.  Applying Theorem \ref{t:st2}, the following can be shown:

\begin{lemma}\label{l:1}
Assume that $A,B\in H_2(M,\mathbb Z)$ are non-trivial with $A=(a_1,a_2)$.  Let $I=\{0,\pm 1,\pm 2\}$ and assume that either $a_1\not\in I$ or $a_2\not\in I$.  If $A$ and $B$ are represented by disjoint connected surfaces, then at most one of the surfaces is an embedded sphere.  
\end{lemma}

\begin{proof}  Denote that $A=(a_1,a_2)$.  As these are represented by disjoint surfaces $\Sigma_A$ and $\Sigma_B$, $A\cdot B=0$, and thus $B=d(a_2, -a_1)$ for some $d\in\mathbb Q$.  

Assume that $A$ is a multiple class and that $\Sigma_A$, $\Sigma_B$ are disjoint spheres. Using the results on multiple classes in Rokhlin \cite{R}, the only multiple classes which can be represented by embedded spheres are precisely the ones in Theorem \ref{t:law}.  Assuming that $a_i\not\in I$, this leaves the classes
\[
A=\pm(3,0)\mbox{  and  }B\in\{(0,\pm 1), (0,\pm 2), (0,\pm 3)\}
\mbox{   
or 
}
A=\pm(0,3)\mbox{  and  }B\in\{(\pm 1,0), (\pm 2,0), (\pm 3,0)\}.
\] 
The pairings such that $A+B$ contains only odd entries cannot be represented by disjoint spheres, as this would imply that either $(3,1)$ or $(3,3)$ are represented by embedded spheres.  However, this is not the case on account of Kervaire-Milnor's congruence and Rokhlin's estimate.
In the remaining cases, apply Theorem \ref{t:st2} to prove the claim; for example, suppose that $A=(3,0)$ and $B=(0,2)$.  Choose $c=(3,1)$, and note that $c\cdot c>\sigma(M)$, that $c\cdot A>0$ and that $c\cdot B>0$.  Then
\[
A\cdot A-c \cdot A= 9-9=0\mbox{  and  }B\cdot B-c\cdot B=4-2=2,
\]
hence both $\chi(\Sigma_a)$ and $\chi(\Sigma_B)$ are non-positive.  The remaining cases are similar.  For example, if $A=(3,0)$ and $B=(0,-2)$, choose $c=(3,-1)$, and the same result follows.

Assume now that $B$ is a multiple class, but that $A$ is primitive.  If $B$ contains a 0, then $A$ contains a $\pm 1$ at the same spot; for example, if $B=(\pm 3,0)$, then $A=(0,\pm 1)$.  Thus the classes with a 0 either lead to a class $A$ with $a_i\in I$ or a $A+B=(3,1)$-type class, which is not represented by an embedded sphere.
This leaves $B=(\pm 2,\pm 2)$.  A calculation similar to that above again shows that both Euler characteristics in Theorem \ref{t:st2} are non-positive. For example, if $B=(2,-2)$, then $A=\pm(1,1)$.  Choose $c=(3,-1)$ or $(1,-3)$; in both cases,
\[
A\cdot A-c \cdot A= 2-2=0\mbox{  and  }B\cdot B-c\cdot B=8-8=0.
\]

Assume finally that both $A$ and $B$ are primitive.  Then $A=(a_1,a_2)$ and $B=\pm(a_2, -a_1)$.  For the initial argument, choose $c=(3,1)$ as the characteristic class and assume that both
\begin{equation}\label{e:est}
c\cdot A=3a_1+a_2\mbox{  and  }c\cdot B=\pm 3a_2\mp a_1
\end{equation}
are non-negative.  Applying Theorem \ref{t:st2}, either
\begin{equation}\label{e:a}
-\chi(\Sigma_A)\ge A\cdot A-c \cdot A= a_1^2-3a_1+a_2^2-a_2
\end{equation}
or
\begin{equation}\label{e:b}
-\chi(\Sigma_B)\ge B\cdot B-c\cdot B=a_1^2\pm a_1+a_2^2\mp 3a_2
\end{equation}
must hold.   Focus now on the right hand side terms:
\begin{equation}\label{e:rhs}
a_1^2-3a_1+a_2^2-a_2\mbox{  and  }a_1^2\pm a_1+a_2^2\mp 3a_2.
\end{equation}

Observe that on $\mathbb Z$, $x^2\pm 3x\ge -2$ and $x^2\pm x\ge 0$.  Moreover, switching between the signs reflects the curve about the y-axis.  Thus, any estimates on $x$ will have their signs reversed under this switch.

It follows that $a_1^2-3a_1\ge -2$ and $a_1^2+\pm a_1\ge 0$ and thus, if $|a_2|\ge 3$, both equations in \ref{e:rhs} are non-negative.  Similarly, it follows that $a_2^2-a_2\ge 0$ and $a_2^2\mp 3a_2\ge -2$.  If $|a_1|\ge 3$, then again both equations in \ref{e:rhs} are non-negative.

This completes the initial argument where both equations in \ref{e:est} are non-negative.  If both are non-positive, then use instead the class $(-3,-1)$ to obtain
\begin{equation}\label{e:rhs2}
\underbrace{a_1^2+3a_1}_{\ge -2}+\underbrace{a_2^2+a_2}_{\ge 0}\mbox{  and  }\underbrace{a_1^2\mp a_1}_{\ge 0}+\underbrace{a_2^2\pm 3a_2}_{\ge -2}.
\end{equation}
These are the same equations one would obtain for $A=(-a_1,-a_2)$ in the argument using \ref{e:rhs}.  Thus, the same result as before holds.

In order to handle the case in which the equations in \ref{e:est} have opposite signs, use the classes $(1,-3)$ and $(-1,3)$.  This is equivalent, in \ref{e:a} and \ref{e:b}, to switching $A$ and $B$ and one pair of signs.  It thus does not change the result obtained above.

More precisely, if $B=\pm (a_2,-a_1)$ and 
\begin{equation}\label{e:est2}
c\cdot A=3a_1+a_2\ge 0\mbox{  and  }c\cdot B=\pm (3a_2-a_1)<0,
\end{equation}
then use $c=\pm (1,-3)$ to obtain
\[
c\cdot A=\pm(a_1-3a_2)>0\mbox{  and  }c\cdot B= a_2+3a_1\ge 0.
\]

Furthermore, if
\begin{equation}\label{e:est3}
c\cdot A=3a_1+a_2< 0\mbox{  and  }c\cdot B=\pm (3a_2-a_1)\ge 0,
\end{equation}
then use $c=\pm(-1,3)$ to obtain
\[
c\cdot A=\pm(-a_1+3a_2)\ge 0\mbox{  and  }c\cdot B=-(a_2+3a_1)> 0.
\]

Consider now the functions arising for $c=(1,-3)$:
\begin{equation}\label{e:rhs3}
a_1^2-a_1+a_2^2+3a_2\mbox{  and  }a_1^2\mp 3a_1+a_2^2\mp a_2.
\end{equation}
These are identical to \ref{e:rhs} for the classes $A=(-a_2,a_1)$ and $B=\pm(a_1,a_2)$.  Hence the same estimates hold.

If we use $c=(-1,3)$, then this is again the same as using the negative of the previous classes in the initial argument, thus applying the initial argument to the classes $(a_2,-a_1)$ and $-d(a_1,a_2)$.

\end{proof}

\begin{lemma}
For any non-trivial class $A= (a_1,a_2)\in H_2(M,\mathbb Z)$ not in the list of Theorem \ref{t:law}, $A$ cannot be represented by two (or more) disjoint essential spheres.
\end{lemma}

\begin{proof}  Observe that no three non-trivial classes can be orthogonal in $M$, as $M$ is definite and has $b^+=2$.  Assume that $A$ can be represented by two disjoint spheres.   Write $A=A_1+A_2$ with $A_1\cdot A_2=0$ and assume that $A_i$ is each represented by an embedded sphere.  
\begin{enumerate}[leftmargin=*]
\item If one of $A_i$ satisfies the conditions of Lemma \ref{l:1}, then any other surface, disjoint from the sphere representing $A_i$, cannot be an embedded sphere.  Thus it is not possible that $A_1$ and $A_2$ are both represented by disjoint embedded spheres.  

\item This leaves that
\[
A_1,A_2\in\{(a_1,a_2)\;|\; |a_i|\le 2\}.
\]
Note that we may remove $(0,0)$ and, as before, the condition on orthogonality implies that if  $A_1=(a,b)$, then $A_2=d(b, -a)$ for some $d\in\mathbb Z$. 

If $d=\pm 1$, then the only classes that can be obtained from this decomposition that are not in the list of Theorem \ref{t:law} are $(\pm 1, \pm 3)$, $(\pm 4,0)$ and $(0,\pm 4)$.  The classes $(\pm 1, \pm 3)$ are not represented by an embedded sphere due to Kervaire-Milnor's congruence for characteristic classes.  For the divisible classes $(\pm 4,0)$ and $(0,\pm 4)$, Rokhlin \cite{R} bounded the genus below by 2.

If $d=\pm 2$, then $|a|,|b|\le 1$, or one of the components of $A_2$ is too large and Lemma \ref{l:1} applies.  Then the only class that can be obtained from this decomposition that is not in the list of Theorem \ref{t:law} is $(\pm 1, \pm 3)$.  As before, these classes are not represented by an embedded sphere.

\end{enumerate}

\end{proof}

\begin{cor} Let $M$ be a smooth, closed, connected, definite four-manifold with $b_1(M)=0$ and $b^+(M)=2$.  If $A$ is not one of the classes in Theorem \ref{t:law}, then it admits no representation by two disjoint essential spheres.
\end{cor}

These results imply a weaker version of Conj. \ref{c:law}.  Note also that this shows that the class $(3,2)\in H_2(\mathbb CP^2\#\mathbb CP^2,\mathbb Z)$ cannot be represented by two disjoint embedded spheres.  It is still unknown if this class can be represented by an embedded sphere.

Nouh \cite{No} studied certain classes of surfaces in $\mathbb CP^2\#\mathbb CP^2$.  The key argument here was to study surfaces in $\mathbb CP^2\backslash B^4$ with a boundary given by a knot $K$.  These were then glued to a disk $\Delta\subset \mathbb CP^2\backslash B^4$ which has $\partial \Delta=K$.  Using this construction, upper bounds were determined for the genus of certain curves; for example if $n\ge 1$, then
\[
g_m((2n+1,0))\le 2n^2-n-1.
\]
Similarly, if $n\ge 2$, then
\[
g_m((2n,0))\le 2n^2-3n.
\]
Furthermore, a minimal genus representative of $(4,\pm 1)$ was constructed using this technique. 

These results lead to the following speculation:  for a class $A=(a_1, a_2)\in H_2(\mathbb CP^2\#\mathbb CP^2,\mathbb Z)$ with $a_i >0$, there exist two disjoint algebraic curves of genus $\binom{a_i-1}{2}$. 

\begin{speculation}
For a configuration of two essential surfaces representing $A$, the sum of the genera is at least $\binom{a_1-1}{2}+\binom{a_2-1}{2}$. 
\end{speculation}

%
%

\section{The Tilted Transport}

In this section, a construction first described in \cite{DLW} is presented.  This turns out to be the circle sum in certain situations.
Let $\pi:(E^{2n},\w_E)\rightarrow D^2$ be a symplectic Lefschetz
fibration.  This means that

\begin{itemize}
  \item $(E,\w_E)$ is a symplectic manifold with boundary $\pi^{-1}(\partial
  D^2)$;
  \item $\pi$ has finitely many critical points $p_0,\dots,p_n$ away
  from $\partial D^2$, while $\pi^{-1}(b)$ is a closed symplectic
  manifold symplectomorphic to $(X,\w)$ when $b\neq \pi(p_i)$ for any
  $i$;
  \item fix a complex structure $j$ on $D^2$.  There is
  another complex structure $J_i$, defined near $p_i$, so that $\pi$
  is $(J_i,j)$-holomorphic in a holomorphic chart $(z_1,\dots,z_n)$ near $p_i$, and under this chart,
  $\pi$ has a local expression $(z_1,\dots, z_n)\mapsto
  z_1^2+\dots+z_n^2$.

\end{itemize}

Take a regular value of $\pi$, $b_0\in D^2$ as the base point.
Suppose that one has a submanifold $Z^{2r-1}\subset \pi^{-1}(b_0)$. Then $Z$ has \textit{isotropic dimension 1} if, at each $x\in Z$,
$(T_xZ)^{\perp\w}\cap T_xZ=\R\langle v_x\rangle$.  We call $v_x$ an
\textit{isotropic vector} at $x$; for example, relevant in what follows, a closed curve on a surface.

Suppose that we have a (based) Lefschetz fibration $(E,\pi, b_0)$ with a
submanifold $Z\subset \pi^{-1}(b_0)$ of isotropic dimension 1.  Let
$\gamma(t)\subset D^2$ be a path with $\gamma(0)=b_0$.  Assume that
$\gamma(t)\neq \pi(p_i)$ for all $t$ and $i$.  Notice that there is a natural
symplectic connection on $E$ in the complement of singular points as
a distribution.  Then, for $x\in E\backslash \coprod_{i=1}^n \{p_i\}$, the
connection at $x$ is defined by $(T^vE)_x^{\perp\w}$ .  Here $T^vE$
is the subbundle of $TE$ defined by vertical tangent spaces
$T(\pi^{-1}(\pi(x)))$ at point $x$.  $E|_\gamma=\pi^{-1}(\gamma)$
  thus inherits this
connection  and thus a trivialization by parallel transports.  The
symplectic connection also defines a unique lift of $\gamma'$ to a
vector field of $E|_\gamma$.  We will use $\pi^{-1}(\gamma')$ to
represent this lift.

Now choose a vector field $V$ on $E|_\gamma$ tangential to the fibers; one
obtains a flow defined by $V+\pi^{-1}(\gamma')$.  Suppose that the
following holds:

\begin{condition}\label{cond:c}\hfill
\begin{itemize}
  \item  $Z_t\subset E_{\gamma(t)}$ is the time $t$-flow of $Z_0=Z$,
  and each $Z_t$ is of isotropic dimension 1.
  \item For any $x_t\in Z_t$, let $v_{x_t}$ be the isotropic vector.  Then $\w(V,v_{x_t})\neq0$.
\end{itemize}

\end{condition}

Let $\widehat Z=\coprod Z_t$.  It is then easy to see that $\wh Z$
is a symplectic submanifold of $E$ with a boundary on $E_{\gamma(0)}$
and $E_{\gamma(1)}$.  We call $\wh Z$ a \textit{tilted transport} of
$Z$.  A special case is when $\gamma(0)=\gamma(1)$ and $Z_0=Z_1$. In
such cases, one could be able to adjust $V$ appropriately so that
$\wh Z$ is a smooth closed symplectic submanifold, which we will
call a \textit{tilted matching cycle}. 

%

\brmk The tilted transport construction as described here can be
easily generalized in many ways.  A most interesting generalization
is that one could admit $V$ with singularities, and thus change the
topology of $Z_t$ when $t$ evolves.  \ermk

\subsubsection{A local variant of tilted transport and
symplectic circle sum}

We explain next how to use a rather simple case of tilted transport
to partly recover the circle sum construction in symplectic
geometry. Note that the corresponding counterpart is well-known in the
smooth category.

The setting under consideration is a pair of disjoint symplectic
surfaces $S_0, S_1\subset (M^4,\w)$.  Suppose that one has an open set
$U\subset M$ so that $U\cong S^1\times [-1,1]\times D^2(2)$, a
trivial bundle over $D^2$ with annulus fibers, while $S_i\cap
U=S^1\times [-1,1]\times \{i\}$.  We claim that there is an embedded
symplectic surface $S$ which is the circle sum of $S_0$ and $S_1$.

The question is local, so we concentrate on the trivial bundle $U$.
Remove the part $S^1\times [-\frac{1}{2},\frac{1}{2}]$ from the
fibers $U_0$ and $U_1$.  Consider two embedded arcs $\gamma(t)$ and
$\bar\gamma(t)$, which only intersect at
$\gamma(0)=\bar\gamma(0)=0$ and $\gamma(1)=\bar\gamma(1)=1$.  By
choosing $V$ appropriately on $U_\gamma$, one easily constructs a
tilted transport which concatenates $S^1\times
[-1,-\frac{1}{2}]\times \{0\}$ with $S^1\times [\frac{1}{2},1]\times
\{1\}$.  Similarly, one concatenates $S^1\times
[-1,-\frac{1}{2}]\times \{1\}$ with $S^1\times [\frac{1}{2},1]\times
\{0\}$ by choosing another tilted transport on $\bar\gamma$.  This
realizes the circle sum, as claimed.

As immediate consequence of the construction, by taking a finite
number of nearby copies of generic fibers in an arbitrary Lefschetz
fibration of dimension $4$, is that one realizes $n[F]$ as an embedded
symplectic surface by performing symplectic circle sums on two
consecutive copies.

{\bf Acknowledgement}  We thank Weiwei Wu for making us aware of Lambert-Cole's results.  The second author is supported by NSF 1611680.

\end{document}